\documentclass[12pt,a4paper]{amsart} 
\usepackage[utf8]{inputenc} 
\usepackage[top=3cm, bottom=3.5cm, left=2.5cm, right=2.5cm]{geometry}
\usepackage{amsmath,amsthm,amsfonts,amstext,amssymb}
\usepackage{mathrsfs}
\usepackage{enumerate}
\usepackage{hyperref}
\usepackage{dsfont}
\usepackage{color}
\usepackage{graphicx}
\usepackage{xcolor}
\hypersetup{
    colorlinks,
    linkcolor={blue!50!black},
    citecolor={blue!50!black},
    urlcolor={blue!80!black}
}

\def\R{{\mathbb{R}}}

\newtheorem{theorem}{Theorem}[section]

\newtheorem{lemma}[theorem]{Lemma}

\theoremstyle{definition}

\numberwithin{equation}{section}

\begin{document}

\title[On the visibility window in continuum percolation models]{On the visibility window for Brownian interlacements, Poisson cylinders and Boolean models}
\author{Yingxin Mu}
\address{
  Yingxin Mu,
  University of Leipzig, Institute of Mathematics,
  Augustusplatz 10, 04109 Leipzig, Germany.
}
\email{yingxin.mu@uni-leipzig.de}

\author{Artem Sapozhnikov}
\address{
  Artem Sapozhnikov,
  University of Leipzig, Institute of Mathematics,
  Augustusplatz 10, 04109 Leipzig, Germany.
}
\email{artem.sapozhnikov@math.uni-leipzig.de}

\begin{abstract}
We study visibility inside the vacant set of three models in $\R^d$ with slow decay of spatial correlations: Brownian interlacements, Poisson cylinders and Poisson-Boolean models. Let $Q_x$ be the radius of the largest ball centered at $x$ every point of which is visible from $0$ through the vacant set of one of these models. We prove that conditioned on $x$ being visible from $0$, $Q_x/\delta_{\|x\|}$ converges weakly, as $x\to\infty$, to the exponential distribution with an explicit intensity, which depends on the parameters of the respective model. The scaling function $\delta_r$ is the visibility window introduced in \cite{MS-visibility}, a length scale of correlations in the visible set at distance $r$ from $0$.
\end{abstract}


\maketitle

\section{Introduction}
Let $\mathcal C$ be a random closed subset of $\R^d$ with a rotationally invariant distribution. 
We think of the set $\mathcal C$ as a random field of non-transparent obstacles and say that a point $x\in\R^d$ is \emph{visible} (from $0$) if the line segment $[0,x]$ does not intersect $\mathcal C$. 
The first mathematical study of visibility goes back to Pólya \cite{Polya-visibility}. The probability of visibility to a large distance was studied in \cite{Calka-visibility} for the obstacles given by a Poisson-Boolean model and in \cite{ET-visibility} for those given by the Brownian interlacements. Sharp bounds on the probability of visibility to a large distance were obtained in \cite{MS-visibility}, when the obstacles are given by the Brownian interlacements, the Poisson cylinders or the Poisson-Boolean model. In the setting of hyperbolic spaces, visibility was studied in 
 \cite{BJST-visibility-H,TC-visibility-H, BHT-hyperbolic}.

 \smallskip

In this paper, we follow up on our study of the visibility through the vacant set of the Brownian interlacements, the Poisson cylinders and the Poisson-Boolean models in \cite{MS-visibility}.
In \cite{MS-visibility}, we introduced the \emph{visibility window} $\delta_r$, a length scale of correlations in the visible set at distance $r$ from $0$, and computed it for the three obstacle models:
\begin{itemize}
\item
(Brownian interlacements)
\[
\delta_r = \delta_\mathrm{BI}(r) = \left\{\begin{array}{ll} r^{-1} & d\geq 4\\[4pt] r^{-1}\log^2r & d=3\end{array}\right.
\]
\item
(Poisson cylinders)
\[
\delta_r = \delta_\mathrm{PC}(r) = \left\{\begin{array}{ll} r^{-1} & d\geq 3\\[4pt] 1 & d=2\end{array}\right.
\]
\item
(Poisson-Boolean models)
\[
\delta_r = \delta_\mathrm{BM}(r) = r^{-1}.
\]
\end{itemize}
Here, we investigate another aspect of the visibility window in these three models. Let $Q_x$ be the radius of the largest ball centered at $x$ every point of which is visible from $0$, 
\begin{equation}\label{def:Qx}
Q_x = \inf\big\{q>0\,:\,\text{every $y\in B(x,q)$ is visible from $0$}\big\}. 
\end{equation}
The main result of this note is that, for each of the three models, conditioned on $x$ being visible from $0$, 
$Q_x/\delta_{\|x\|}$ converges weakly, as $x\to\infty$, to the exponential distribution with intensity $\lambda>0$, which explicitly depends on the parameters of the respective model. 

\medskip

We now briefly describe the three models and refer to seminal papers \cite{Sznitman-BI,TW-cylinders} and to \cite{MR-Book} for their precise descriptions in terms of Poisson point processes. 
In fact, for the purpose of this article, we only need the characterization of their laws by the functional\footnote{For more on this topic, we refer to \cite[Chapter~2]{Matheron}.}
\[
T(K) = \mathsf P[\mathcal C\cap K\neq\emptyset],\quad\text{for compact }K\subset\R^d.
\]

\smallskip

\emph{Brownian interlacements:} For $\alpha>0$, let $\mathcal I^\alpha$ be the range of a Poisson soup of doubly infinite Brownian motions in $\R^d$ ($d\geq 3$) with intensity $\alpha$. It is a random closed subset of $\R^d$, whose law is characterized by the relations 
\begin{equation}\label{eq:BI-capacity}
\mathsf P \big[\mathcal I^\alpha\cap K = \emptyset\big] = e^{-\alpha\mathrm{cap}(K)},\quad\text{for compact }K\subset\R^d,
\end{equation}
(see \cite[Proposition~2.5]{Sznitman-BI}) where $\mathrm{cap}(K)$ is the Newtonian capacity of $K$, see \eqref{def:equilibrium-measure-and-capacity}. 
The \emph{Brownian interlacements at level $\alpha$ with radius $\rho$} is the closed $\rho$-neighborhood of $\mathcal I^\alpha$, 
\begin{equation}\label{def:BI-alpha-rho}
\mathcal I^\alpha_\rho = 
\bigcup\limits_{x\in\mathcal I^\alpha} B(x,\rho).
\end{equation}

 \emph{Poisson cylinders:} For $\alpha>0$, let $\mathcal L^\alpha$ be the range of a Poisson soup of lines in $\R^d$ ($d\geq 2$) with intensity $\alpha$. To describe the law of $\mathcal L^\alpha$, let $\nu$ be the unique Haar measure on the topological group $SO_d$ of rigid rotations in $\R^d$ with $\nu(SO_d)=1$ and denote by $\pi$ the orthogonal projection on the hyperplane $\{x=(x_1,\ldots, x_d)\in\R^d\,:\,x_1=0\}$. 
 By \cite[(2.8)]{TW-cylinders}, the law of the random closed set $\mathcal L^\alpha$ is characterized by the relations 
\begin{equation}\label{def:PC-PL}
\mathsf P \big[\mathcal L^\alpha\cap K = \emptyset\big] = e^{-\alpha\mu(K)},\quad\text{for compact }K\subset\R^d,
\end{equation}
where 
\begin{equation}\label{def:PC-mu}
\mu(K) = \int_{SO_d}\lambda_{d-1}\big(\pi(\phi(K))\big)\,\nu(d\phi).
\end{equation}
The \emph{Poisson cylinders model at level $\alpha$ with radius $\rho$} is the closed $\rho$-neighborhood of $\mathcal L^\alpha$, 
\begin{equation}\label{def:PC-alpha-rho}
\mathcal L^\alpha_\rho = 
\bigcup\limits_{x\in\mathcal L^\alpha} B(x,\rho).
\end{equation}

\emph{Poisson-Boolean models:} 
For $\alpha>0$ and a probability distribution $\mathsf Q$ on $\R_+$, 
let $\omega = \sum_{i\geq 1}\delta_{(x_i,r_i)}$ be a Poisson point process on $\R^d\times\R_+$ ($d\geq 2$) with intensity measure $\alpha dx\otimes \mathsf Q$. 
The \emph{Poisson-Boolean model with intensity $\alpha$ and radii distribution $\mathsf Q$} is the closed subset of $\R^d$, defined as 
\[
\mathcal B^\alpha_\mathsf Q = \mathcal B^\alpha_\mathsf Q(\omega) =\bigcup\limits_{i\geq 1}B(x_i,r_i).
\]
This set does not coincide with $\R^d$ if and only if 
\begin{equation}\label{eq:BM-expected-volume}
\mathsf E[\varrho^d]<\infty,
\end{equation}
where $\varrho$ is a random variable with law $\mathsf Q$ (see \cite[Proposition~3.1]{MR-Book}).
The number of balls that intersect a compact set $K$ is a Poisson random variable with parameter 
\[
\int\limits_{(x,r)\,:\,B(x,r)\cap K\neq\emptyset}\alpha dx\otimes \mathsf Q(dr) = \alpha\int\limits_{\R^d}\mathsf P\big[\varrho\geq d(x,K)\big]\,dx = \alpha\,\mathsf E\big[\lambda_d\big(B(K,\varrho)\big)\big], 
\]
where $B(K,\varrho)$ is the closed $\varrho$-neighborhood of $K$; in particular, the law of $\mathcal B^\alpha_\mathsf Q$ is characterized by the relations
\begin{equation}\label{def:BM}
\mathsf P \big[\mathcal B^\alpha_\mathsf Q\cap K = \emptyset\big] = e^{-\alpha\,\mathsf E[\lambda_d(B(K,\varrho))]},\quad\text{for compact }K\subset\R^d.
\end{equation}

\medskip

We now state our main result precisely. 

\begin{theorem}\label{thm:main-result}
Let $\alpha>0$ and $\rho>0$, let $\mathsf Q$ be a probability measure on $\R_+$. 
Let $\mathcal C$ be either the Brownian interlacements at level $\alpha$ with radius $\rho$ ($d\geq 3$) or the Poisson cylinders at level $\alpha$ with radius $\rho$ ($d\geq 2$) or the Poisson-Boolean model with intensity $\alpha$ and radius distribution $\mathsf Q$ ($d\geq 2$) satisfying \eqref{eq:BM-expected-volume}. 
There exists $\lambda>0$, which depends on the model and on the parameters of the model such that, conditioned on $x$ being visible from $0$, 
\[
\text{$\frac{Q_x}{\delta_{\|x\|}}$ converges weakly to the exponential distribution with intensity $\lambda$,}
\]
 as $x\to\infty$, where $\delta_r$ is the respective visibility window of the model. The parameter $\lambda$ is explicit for each of the models, see \eqref{eq:BI-lambda}, \eqref{eq:BM-lambda} and \eqref{eq:PC-lambda}.
\end{theorem}
We prove Theorem~\ref{thm:main-result} separateley for each of the three models: for the Brownian interlacements in Section~\ref{sec:BI}, for the Poisson-Boolean model in Section~\ref{sec:BM} and for the Poisson cylinders in Section~\ref{sec:PC}.

\medskip

We finish this introduction by fixing some common notation, used throughout the proofs. 
Let $x\in\R^d$, $q>0$ and $K\subset \R^d$. We denote by $B(x,q)$ the closed Euclidean ball at $x$ with radius $q$ and write $B(q)=B(0,q)$. We denote the closed $q$-neighborhood of $K$ by $B(K,q)$, that is $B(K,q) = \bigcup_{x\in K}B(x,q)$. 
We denote by $\ell_x$ the line segment $[0,x]$ in $\R^d$ and define $\ell_x(q) = B(\ell_x,q)$. Furthermore, for $\epsilon>0$, we define 
\begin{equation}\label{def:cones}
\ell_x^\epsilon = \bigcup_{y\in B(x,\epsilon)}\ell_y
\quad\text{and}\quad
\ell_x^\epsilon(q) = B(\ell_x^\epsilon,q) = \bigcup_{y\in B(x,\epsilon)}\ell_y(q).
\end{equation}
Finally, we denote by $\lambda_n$ the $n$-dimensional Lebesgue measure and by $\kappa_n$ the volume of the unit ball in $\R^n$ (recall that $\kappa_n = \frac{\pi^{n/2}}{\Gamma(\frac{n}{2} + 1)}$).

\section{Proof of Theorem~\ref{thm:main-result} for the Brownian interlacements}\label{sec:BI}

We begin this section with some preliminaries on Brownian motion and capacity and the two key ingredients for the proof of Theorem~\ref{thm:main-result}---Lemma~\ref{l:capacity-cylinder} about the exact asymptotics for the capacity of a long cylinder and Lemma~\ref{l:bounds-non-hitting} about sharp bounds on the non-hitting probability of a long cylinder by the Brownian motion. The proof of Theorem~\ref{thm:main-result} is given in Section~\ref{sec:BI-proof}.

\subsection{Brownian motion and potential theory}

Let $W$ be a Brownian motion in $\R^d$. We denote by $\mathsf P_x$ the law of $W$ with $W_0=x$ and we write $\mathsf P_\nu$ for $\int_{\R^d}\mathsf P_x[\cdot]\nu(dx)$. For a closed set $K\subset\R^d$, let $H_K=\inf\{t\geq 0\,:\,W_t\in K\}$ be the first entrance time of $W$ in $K$. 
Classically, for any $R_1<R_2$ and $y\in\R^d$ with $R_1<\|y\|<R_2$, 
\begin{equation}\label{eq:BM-hitting}
\mathsf P_y\big[H_{\partial B(R_2)}<H_{\partial B(R_1)}\big] = 
\left\{\begin{array}{ccl}\frac{\log R_1 - \log \|y\|}{\log R_1 - \log R_2} &\quad& d=2\\[10pt] 
\frac{R_1^{2-d}-\|y\|^{2-d}}{R_1^{2-d}-R_2^{2-d}} &\quad& d\geq 3\end{array}\right.
\end{equation}
(see e.g.\ \cite[Theorem~3.18]{MP-BM-book}), in particular, when $d\geq 3$, 
\begin{equation}\label{eq:BM-escape}
\mathsf P_y\big[H_{\partial B(R_1)}=\infty\big] = 1 - \tfrac{\|y\|^{2-d}}{R_1^{2-d}}\,.
\end{equation}

\smallskip

Let $\sigma_R$ be the uniform distribution on $\partial B(R)$ and define $\mu_R=\tfrac{2\pi^{d/2}R^{d-2}}{\Gamma(d/2-1)}\sigma_R$. 
For any compact set $K$ in $\R^d$ ($d\geq 3$) and any $R$ such that $K\subset B(R)$, 
the equilibrium measure $\mu_K$ and the capacity $\mathrm{cap}(K)$ of $K$ are defined as 
(see \cite[Theorem~3.1.10]{PortStone})
\begin{equation}\label{def:equilibrium-measure-and-capacity}
\mu_K = \mathsf P_{\mu_R}\big[H_K<\infty, W_{H_K}\in\cdot\big]\quad\text{resp.}\quad
\mathrm{cap}(K) = \mu_K(K) = \mathsf P_{\mu_R}\big[H_K<\infty\big].
\end{equation}
Hence, by the strong Markov property, for any compacts $K'\subseteq K$, 
\begin{equation}\label{def:equilibrium-measure-and-capacity-2}
\mu_{K'}= \mathsf P_{\mu_K}\big[H_{K'}<\infty, W_{H_{K'}}\in\cdot\big]\quad\text{and}\quad
\mathrm{cap}(K') = \mathsf P_{\mu_K}\big[H_{K'}<\infty\big];
\end{equation}
in particular, $\mu_K(K')\leq \mathrm{cap}(K')$. 
Capacity is an invariant under isometries and monotone function on compacts, $\mathrm{cap}(\rho K) = \rho^{d-2}\mathrm{cap}(K)$ and $\mathrm{cap}(K) = \mathrm{cap}(\partial K)$ (see \cite[Proposition~3.1.11]{PortStone}). 

\smallskip

By \cite[Proposition~3.3.4]{PortStone}, 
there exist $c_i=c_i(d,\rho)$, such that for all $x\in\R^d$ with $\|x\|\geq 2$, the capacity of the cylinder $\ell_x(\rho)$ satisfies 
\begin{equation}\label{eq:capacity-cylinder-bounds}
c_1\big(\mathds{1}_{d=3}\tfrac{\|x\|}{\log \|x\|} + \mathds{1}_{d\geq 4}\|x\|\big)
\leq 
\mathrm{cap}\big(\ell_x(\rho)\big)
\leq 
c_2\big(\mathds{1}_{d=3}\tfrac{\|x\|}{\log \|x\|} + \mathds{1}_{d\geq 4}\|x\|\big).
\end{equation}

\medskip

In the next lemma, we obtain the exact asymptotics for the capacity of the cylinder $\ell_x(\rho)$. 
\begin{lemma}\label{l:capacity-cylinder}
For any $d\geq 3$, 
\[
\mathrm{cap}\big(\ell_x(\rho)\big) = \varkappa_d\rho^{d-3}\big(\mathds{1}_{d=3}\tfrac{\|x\|}{\log \|x\|} + \mathds{1}_{d\geq 4}\|x\|\big)(1+o(1)),\quad\text{as }x\to\infty,
\]
where $\varkappa_3= \pi$ and $\varkappa_d = \frac{2\pi^{\frac{d-1}2}}{\Gamma(\frac{d-3}{2})}$ for $d\geq 4$. 
\end{lemma}
\begin{proof}
Since $\mathrm{cap}\big(\ell_x(\rho)\big) = \rho^{d-2}\mathrm{cap}\big(\ell_{\rho^{-1}\|x\|e_1}(1)\big)$, it suffices to prove the lemma for $x=re_1$ and $\rho=1$. Let $K_r=\ell_{re_1}(1)$ and $k_r = \ell_{re_1}$. Let $r\geq 2$. 

\smallskip

By \cite[Theorem~3.2.1]{PortStone}, for every $y\in k_r$, 
\[
1 = \int_{\partial K_r} G(y,z) \mu_{K_r}(dz), 
\]
where $G$ is the Green function for the standard Brownian motion, 
\[
G(y,z) = \gamma_d\|y-z\|^{2-d}, \quad\text{with }\gamma_d=\frac{\Gamma(\frac{d-2}2)}{2\pi^{d/2}}
\]
(see e.g.\ \cite[(3.1)]{PortStone}).
Integration over the segment $k_r$ gives
\[
r = \gamma_d\int_{\partial K_r} \mu_{K_r}(dz)\,\int_{k_r}\|y-z\|^{2-d}dy.
\]
Let $K_r' = \big\{z\in K_r:z_1\in\big[\tfrac{r}{\log r},r-\tfrac{r}{\log r}\big]\big\}$ and $K_r'' = K_r\setminus K_r'$. 
Note that 
\begin{itemize}\itemsep4pt
\item
there exists $C=C(d)$ such that for all $z\in\partial K_r$, 
\[
\int_{k_r}\|y-z\|^{2-d}dy \leq C\big(\mathds{1}_{d=3}\log r + \mathds{1}_{d\geq 4}\big);
\]
\item
uniformly over $z\in \partial K_r\cap K_r'$, 
\begin{eqnarray*}
\int_{k_r}\|y-z\|^{2-d}dy &= 
&\int_0^{z_1}(1+u^2)^{\frac{2-d}2}du + 
\int_0^{r-z_1}(1+u^2)^{\frac{2-d}2}du\\
&= &2(1+o(1))\big(\mathds{1}_{d=3}\log r + \mathds{1}_{d\geq 4}
\int_0^\infty(1+u^2)^{\frac{2-d}2}du\big)\\
&= &2\beta_d(1+o(1))\big(\mathds{1}_{d=3}\log r + \mathds{1}_{d\geq 4}\big),
\end{eqnarray*}
where $\beta_3=1$ and $\beta_d = \frac12\mathrm{Beta}(\tfrac12,\tfrac{d-3}{2}) = \frac{\sqrt{\pi}\Gamma(\frac{d-3}2)}{2\Gamma(\frac{d-2}2)}$ for $d\geq 4$, see \cite[6.2.1]{AS:math-functions};

\item
by \eqref{eq:capacity-cylinder-bounds}, there exists $C'=C'(d)$ such that 
\[
\mu_{K_r}(\partial K_r\cap K_r'') \leq \mathrm{cap}(K_r'')
\leq C'\big(\mathds{1}_{d=3}\tfrac{r}{\log^2r} + \mathds{1}_{d\geq 4}\tfrac{r}{\log r}\big).
\]
\end{itemize}
Hence 
\[
r = 2\beta_d\gamma_d(1+o(1))\big(\mathds{1}_{d=3}\log r + \mathds{1}_{d\geq 4}\big)\mathrm{cap}(K_r) + o(r),
\]
as $r\to\infty$, and the result follows with $\varkappa_d = (2\beta_d\gamma_d)^{-1}$. 
\end{proof}

\medskip

In the next lemma, we obtain sharp bounds for non-hitting probability of a cylinder. 
For $x=(x_1,\ldots, x_d)\in\R^d$, let $\widetilde x = (x_2,\ldots, x_d)\in\R^{d-1}$. 
\begin{lemma}\label{l:bounds-non-hitting}
Let $d\geq 3$ and $\rho>0$. For any $\varepsilon>0$, there exist $r_0=r_0(d,\rho,\varepsilon)<\infty$ and $\delta_0=\delta_0(d,\rho,\varepsilon)\in(0,1)$, such that for all $r\geq r_0$ and $x\in\R^d$ with $x_1\in\big[\frac{r}{\log^2 r},r-\frac{r}{\log^2r}\big]$ and $\rho< \|\widetilde x\|< \rho+\delta_0$, 
\[
1-\varepsilon\leq 
\frac{\mathsf P_x\big[H_{\ell_{re_1}(\rho)}=\infty\big]}
{\tfrac{\|\widetilde x\|-\rho}{\rho}\big(\mathds{1}_{d=3}\tfrac{1}{\log r} + \mathds{1}_{d\geq 4}(d-3)\big)}
\leq 
1+\varepsilon.
\]
\end{lemma}
\begin{proof}
Let $\varepsilon>0$. Let $r$ be large enough and choose $x$ as in the statement of the lemma. 
For $s>0$, let $L(s)=B(\R e_1,s)$ be the doubly infinite cylinder with radius $s$ and axis vector $e_1$. 

\smallskip

We begin with the lower bound. By the strong Markov property, 
\[
\mathsf P_x\big[H_{\ell_{re_1}(\rho)}=\infty\big]  \geq 
\mathsf P_x\big[H_{\partial L(r\log r)}<H_{L(\rho)}\big]\,
\inf\limits_{y\in \partial L(r\log r)}\mathsf P_y\big[H_{\ell_{re_1}(\rho)}=\infty\big].
\]
Note that the orthogonal projection of the Brownian motion $W$ onto the hyperplane $\{x\in\R^d:x_1=0\}$ is a standard $(d-1)$-dimensional Brownian motion and the projection of the cylinder $L(s)$ on the hyperplane is the $(d-1)$-dimensional Euclidean ball of radius $s$. Thus, by \eqref{eq:BM-hitting}, if $r$ is large enough and $\|\widetilde x\|-\rho$ small enough, then 
\[
\mathsf P_x\big[H_{\partial L(r\log r)}<H_{L(\rho)}\big]\geq 
(1-\tfrac12\varepsilon)\tfrac{\|\widetilde x\|-\rho}{\rho}\big(\mathds{1}_{d=3}\tfrac{1}{\log r} + \mathds{1}_{d\geq 4}(d-3)\big).
\]
Furthermore, by \eqref{eq:BM-escape}, if $r$ is large enough, then 
\[
\inf\limits_{y\in \partial L(r\log r)}\mathsf P_y\big[H_{\ell_{re_1}(\rho)}=\infty\big]
\geq 
\inf\limits_{y\in \partial L(r\log r)}\mathsf P_y\big[H_{B(r+\rho)}=\infty\big]
\geq 1-\tfrac12\varepsilon.
\]
Putting the two bounds together gives the desired lower bound. 

\smallskip

We proceed with the upper bound. Let $s_r=\frac{r}{\log^4 r}$. We have 
\[
\mathsf P_x\big[H_{\ell_{re_1}(\rho)}=\infty\big]  \leq 
\mathsf P_x\big[H_{\partial L(s_r)}<H_{L(\rho)}\big] 
+ 
\mathsf P_x\big[H_{\ell_{re_1}(\rho)}=\infty, H_{\partial L(s_r)}>H_{L(\rho)}\big].
\]
As in the proof of the lower bound, by \eqref{eq:BM-hitting}, if $r$ is large enough and $\|\widetilde x\|-\rho$ small enough, then 
\[
\mathsf P_x\big[H_{L(s_r)}<H_{L(\rho)}\big]\leq 
(1+\tfrac12\varepsilon)\tfrac{\|\widetilde x\|-\rho}{\rho}\big(\mathds{1}_{d=3}\tfrac{1}{\log r} + \mathds{1}_{d\geq 4}(d-3)\big).
\]
To bound the second probability, note that by the rotational symmetry of the Brownian motion and \eqref{eq:BM-escape}, for any $y\in L(s_r)$, the probability that the Brownian motion started at $y$ will exit from the ball $B(y_1e_1, 2s_r)$ before hitting $B(y_1e_1,\rho)$ or $\partial L(s_r)$ is bounded from above by $\min\big(C\frac{(\|\widetilde y\|-\rho)_+}{\rho}, 1-c\big)$ for some $C=C(d)<\infty$ and $c=c(d)\in(0,1)$. If the event in the second probability occurs, since $x_1\in [s_r\log^2r, r-s_r\log^2r]$, the Brownian motion started at $x$ will exit the ball $B(x_1e_1,s_r\log^2 r)$ before hitting $B(x_1e_1,\rho)$ or $\partial L(s_r)$. Thus, by the strong Markov property at the exit times from the balls $B(x_1e_1,2s_rk)$, $1\leq k\leq\frac12\log^2r-1$,  
\begin{eqnarray*}
\mathsf P_x\big[H_{\ell_{re_1}(\rho)}=\infty, H_{\partial L(s_r)}>H_{L(\rho)}\big]
&\leq &C'\tfrac{\|\widetilde x\|-\rho}{\rho}(1-c)^{\frac12\log^2r}\\
&<&\tfrac12\varepsilon\tfrac{\|\widetilde x\|-\rho}{\rho}\big(\mathds{1}_{d=3}\tfrac{1}{\log r} + \mathds{1}_{d\geq 4}(d-3)\big),
\end{eqnarray*}
for some $C'=C'(d)<\infty$ and $c\in(0,1)$ as above. 
Putting the bounds together gives the desired upper bound. 
\end{proof}

\subsection{Proof of Theorem~\ref{thm:main-result}}\label{sec:BI-proof}

Fix positive $\alpha$ and $\rho$ and denote by $\mathsf P_{\alpha,\rho}$ the law of the Brownian interlacements at level $\alpha$ with radius $\rho$. 
Recall that $\delta_r = \frac{\log^2r}{r}$ for $d=3$ and $\delta_r = \frac{1}{r}$ for $d\geq 4$. We aim to prove that for every $s>0$, 
\[
\lim\limits_{r\to\infty}\mathsf P_{\alpha,\rho}\big[Q_{re_1}>s\delta_r\,|\,re_1\text{ is visible from }0\big] = \exp\big(-\lambda_{\mathrm{BI}}s\big)
\]
with
\begin{equation}\label{eq:BI-lambda}
\lambda_{\mathrm{BI}} = \frac12\alpha\varkappa_d\rho^{d-4}\big(\mathds{1}_{d=3} + \mathds{1}_{d\geq 4}(d-3)\big)
= \left\{\begin{array}{ll}\frac{\alpha\pi}{2\rho} & d=3\\[4pt]
\frac{\alpha\rho^{d-4}\pi^{\frac{d-1}{2}}}{\Gamma(\frac{d-3}{2})} & d\geq 4\end{array}\right.
\end{equation}
where $\varkappa_d$ is the constant from Lemma~\ref{l:capacity-cylinder}.

\medskip

By the definition \eqref{def:Qx} of $Q_{re_1}$, \eqref{eq:BI-capacity}, \eqref{def:BI-alpha-rho} and \eqref{def:cones}, 
\begin{multline*}
\mathsf P_{\alpha,\rho}\big[Q_{re_1}>s\delta_r\,|\,re_1\text{ is visible from }0\big]
= 
\frac{\mathsf P_{\alpha,\rho}[\text{every $x\in B(re_1,s\delta_r)$ is visible from $0$}]}{\mathsf P_{\alpha,\rho}[\text{$re_1$ is visible from $0$}]}\\
=
\frac{\mathsf P [\mathcal I^\alpha_\rho\cap \ell_{re_1}^{s\delta_r} = \emptyset]}{\mathsf P[\mathcal I^\alpha_\rho\cap \ell_{re_1} = \emptyset]}
=
\frac{\mathsf P [\mathcal I^\alpha\cap \ell_{re_1}^{s\delta_r}(\rho) = \emptyset]}{\mathsf P[\mathcal I^\alpha\cap \ell_{re_1}(\rho) = \emptyset]}
=
\exp\big(-\alpha\big[\mathrm{cap}\big(\ell_{re_1}^{s\delta_r}(\rho)\big) - \mathrm{cap}\big(\ell_{re_1}(\rho)\big)\big]\big).
\end{multline*}
Thus, it suffices to prove that 
\[
\lim\limits_{r\to\infty}\big[\mathrm{cap}\big(\ell_{re_1}^{s\delta_r}(\rho)\big) - \mathrm{cap}\big(\ell_{re_1}(\rho)\big)\big] = \frac{\lambda_{\mathrm{BI}}}{\alpha}s.
\]
We write $\mu_r$ resp.\ $\mu_r^s$ for the equilibrium measure of $\ell_{re_1}(\rho)$ resp.\ $\ell_{re_1}^{s\delta_r}(\rho)$. 
By \eqref{def:equilibrium-measure-and-capacity-2}, since $\ell_{re_1}(\rho)\subset\ell_{re_1}^{s\delta_r}(\rho)$, 
\[
\mathrm{cap}\big(\ell_{re_1}^{s\delta_r}(\rho)\big) - \mathrm{cap}\big(\ell_{re_1}(\rho)\big)
= \int_{\partial \ell_{re_1}^{s\delta_r}(\rho)}\mathsf P_x[H_{\ell_{re_1}(\rho)}=\infty]\,\mu_r^s(dx).
\]
Let 
\[
\partial' = \big\{x\in\partial \ell_{re_1}^{s\delta_r}(\rho)\,:\, \tfrac{r}{\log^2r}\leq x_1\leq r - \tfrac{r}{\log^2r}\big\}\quad \text{and}\quad \partial'' = \partial \ell_{re_1}^{s\delta_r}(\rho)\setminus \partial'.
\]

Since every $x\in \partial \ell_{re_1}^{s\delta_r}(\rho)$ is at distance at most $s\delta_r$ from a ball $B_x$ of radius $\rho$ contained in $\ell_{re_1}(\rho)$, by \eqref{eq:BM-escape}, 
\[
\int_{\partial''}\mathsf P_x[H_{\ell_{re_1}(\rho)}=\infty]\,\mu_r^s(dx)
\leq 
\int_{\partial''}\mathsf P_x[H_{B_x}=\infty]\,\mu_r^s(dx)
\leq
Cs\delta_r\,\mu_r^s(\partial'') \leq Cs\delta_r\,\mathrm{cap}(\partial ''),
\]
for some $C=C(d,\rho)$. By \eqref{eq:capacity-cylinder-bounds}, $\delta_r\,\mathrm{cap}(\partial '')\to 0$ as $r\to\infty$. 
Thus, it suffices to prove that 
\[
\lim\limits_{r\to\infty} \int_{\partial'}\mathsf P_x[H_{\ell_{re_1}(\rho)}=\infty]\,\mu_r^s(dx)
= \frac{\lambda_{\mathrm{BI}}}{\alpha}s.
\]
By Lemma~\ref{l:bounds-non-hitting}, since for every $x\in\partial'$, $\|\widetilde x\|\leq \rho + s\delta_r$ for all large $r$, 
\begin{eqnarray*}
\lim\limits_{r\to\infty} \int_{\partial'}\mathsf P_x[H_{\ell_{re_1}(\rho)}=\infty]\,\mu_r^s(dx)
&=
&\lim\limits_{r\to\infty}\frac{1}{\rho}\big(\mathds{1}_{d=3}\tfrac{1}{\log r} + \mathds{1}_{d\geq 4}(d-3)\big)
\int_{\partial'}(\|\widetilde x\|-\rho)\,\mu_r^s(dx)\\
&=
&\lim\limits_{r\to\infty}\frac{1}{\rho}\big(\mathds{1}_{d=3}\tfrac{1}{\log r} + \mathds{1}_{d\geq 4}(d-3)\big)
\int_{\partial'}\big(\frac{s\delta_r}{r}x_1\big)\,\mu_r^s(dx).
\end{eqnarray*}
Now,
\[
\int_{\partial'}x_1\,\mu_r^s(dx) = 
\int_{\partial'}\big(\int_0^{x_1}dy\big)\,\mu_r^s(dx) 
= \int_0^r\mu_r^s\big(x\in\partial' : x_1>y\big)dy.
\]
Thus, 
\[
\frac{\delta_r}{r}\Big|\int_{\partial'}x_1\,\mu_r^s(dx) - \int_0^r\mu_r^s\big(x\in\ell_{re_1}^{s\delta_r}(\rho) : x_1>y\big)dy\Big|
\leq 
\frac{\delta_r}{r}\,r\mathrm{cap}(\partial'')\to 0,\quad\text{as }r\to\infty,
\]
and it suffices to prove that 
\[
\lim\limits_{r\to\infty}\frac{\delta_r}{r}\big(\mathds{1}_{d=3}\tfrac{1}{\log r} + \mathds{1}_{d\geq 4}(d-3)\big)
\int_0^r\mu_r^s\big(x\in\ell_{re_1}^{s\delta_r}(\rho) : x_1>y\big)dy = \frac{\lambda_{\mathrm{BI}}\rho}{\alpha}.
\]
Note that by the symmetry of $\mu_r$, 
\[
\int_0^r\mu_r\big(x\in\ell_{re_1}(\rho)\,:\,x_1>y\big)dy
=
\int_0^r\mu_r\big(x\in\ell_{re_1}(\rho)\,:\,x_1<y\big)dy 
=\frac{r}{2}\mathrm{cap}\big(\ell_{re_1}(\rho)\big),
\]
and by Lemma~\ref{l:capacity-cylinder} and the definition of $\delta_r$, 
\[
\lim\limits_{r\to\infty}\frac{\delta_r}{r}\big(\mathds{1}_{d=3}\tfrac{1}{\log r} + \mathds{1}_{d\geq 4}(d-3)\big)\,
\frac{r}{2}\mathrm{cap}\big(\ell_{re_1}(\rho)\big) 
= 
\frac12\varkappa_d\rho^{d-3}\big(\mathds{1}_{d=3} + \mathds{1}_{d\geq 4}(d-3)\big)
= 
\frac{\lambda_{\mathrm{BI}}\rho}{\alpha}.
\]
Hence, to complete the proof, if suffices to show that 
\begin{equation}\label{eq:BI:capacities-difference-2}
\lim\limits_{r\to\infty}\frac{\delta_r}{r}
\int_0^r\Big(\mu_r^s\big(x\in\ell_{re_1}^{s\delta_r}(\rho) : x_1>y\big) - \mu_r\big(x\in\ell_{re_1}(\rho) : x_1>y\big)\Big)dy = 0.
\end{equation}
We have 
\begin{multline*}
\int_0^r\Big(\mu_r^s\big(x\in\ell_{re_1}^{s\delta_r}(\rho) : x_1>y\big) - \mu_r\big(x\in\ell_{re_1}(\rho) : x_1>y\big)\Big)dy\\
\leq 2\rho\,\mathrm{cap}\big(\ell_{re_1}^{s\delta_r}(\rho)\big) + 
\int_0^{r-2\rho}\Big(\mu_r^s\big(x\in\ell_{re_1}^{s\delta_r}(\rho) : x_1>y+2\rho\big) - \mu_r\big(x\in\ell_{re_1}(\rho) : x_1>y\big)\Big)dy.
\end{multline*}
By the definition \eqref{def:equilibrium-measure-and-capacity-2} of the equilibrium measure,
\begin{multline*}
\mu_r^s\big(x\in\ell_{re_1}^{s\delta_r}(\rho) : x_1>y+2\rho\big) - \mu_r\big(x\in\ell_{re_1}(\rho) : x_1>y\big)\\
\leq 
\int_{\partial \ell_{re_1}^{s\delta_r}(\rho)}\mathds{1}_{x_1>y+2\rho}\,
\mathsf P_x\big[H_{\ell_{re_1}(\rho)}<\infty, (W_{H_{\ell_{re_1}(\rho)}})_1\leq y\big]\,\mu_r^s(dx).
\end{multline*}
Note that for every $x\in\partial \ell_{re_1}^{s\delta_r}(\rho)$ with $x_1>y+2\rho$, there exists a ball $B_x\subset \ell_{re_1}(\rho)$ of radius $\rho$ at distance at most $s\delta_r$ from $x$, such that if a Brownian motion started from $x$ hits $\ell_{re_1}(\rho)$ for the first time at a point with $x_1$-coordinate $\leq y$, then the Brownian motion must exit from the $\rho$-neighborhood of $B_x$ before hitting $B_x$. Thus, by \eqref{eq:BM-hitting}, the probability under the integral is bounded from above by $Cs\delta_r$ for some $C=C(d,\rho)$. 
Hence 
\[
\int_0^r\Big(\mu_r^s\big(x\in\ell_{re_1}^{s\delta_r}(\rho) : x_1>y\big) - \mu_r\big(x\in\ell_{re_1}(\rho) : x_1>y\big)\Big)dy
\leq \big(2\rho + Cs\delta_r r\big)\,\mathrm{cap}\big(\ell_{re_1}^{s\delta_r}(\rho)\big).
\]
Similarly, 
\begin{multline*}
\int_0^r\Big(\mu_r^s\big(x\in\ell_{re_1}^{s\delta_r}(\rho) : x_1>y\big) - \mu_r\big(x\in\ell_{re_1}(\rho) : x_1>y\big)\Big)dy\\
\geq -2\rho\,\mathrm{cap}\big(\ell_{re_1}(\rho)\big) + 
\int_0^{r-2\rho}\Big(\mu_r^s\big(x\in\ell_{re_1}^{s\delta_r}(\rho) : x_1>y\big) - \mu_r\big(x\in\ell_{re_1}(\rho) : x_1>y+2\rho\big)\Big)dy,
\end{multline*}
where, by \eqref{def:equilibrium-measure-and-capacity-2} and \eqref{eq:BM-hitting},
\begin{multline*}
\mu_r^s\big(x\in\ell_{re_1}^{s\delta_r}(\rho) : x_1>y\big) - \mu_r\big(x\in\ell_{re_1}(\rho) : x_1>y+2\rho\big)\\
\geq 
-\int_{\partial \ell_{re_1}^{s\delta_r}(\rho)}\mathds{1}_{x_1\leq y}\,
\mathsf P_x\big[H_{\ell_{re_1}(\rho)}<\infty, (W_{H_{\ell_{re_1}(\rho)}})_1> y+2\rho\big]\,\mu_r^s(dx)
\geq 
-Cs\delta_r\,\mathrm{cap}\big(\ell_{re_1}^{s\delta_r}(\rho)\big).
\end{multline*}
Thus, 
\[
\Big|\int_0^r\Big(\mu_r^s\big(x\in\ell_{re_1}^{s\delta_r}(\rho) : x_1>y\big) - \mu_r\big(x\in\ell_{re_1}(\rho) : x_1>y\big)\Big)dy\Big|
\leq \big(2\rho + Cs\delta_r r\big)\,\mathrm{cap}\big(\ell_{re_1}^{s\delta_r}(\rho)\big),
\]
and \eqref{eq:BI:capacities-difference-2} follows from Lemma~\ref{l:capacity-cylinder} and the definition of $\delta_r$.
The proof is completed. 
\qed

\section{Proof of Theorem~\ref{thm:main-result} for the Poisson-Boolean model}\label{sec:BM}

Fix $\alpha>0$ and a probability distribution $\mathsf Q$ on $\R_+$ and denote by $\mathsf P_{\alpha,\mathsf Q}$ the law of the Poisson-Boolean model with intensity $\alpha$ and radii distribution $\mathsf Q$. Recall that $\delta_r = \frac{1}{r}$ for all $d\geq 2$. 
We aim to prove that for every $s>0$, 
\[
\lim\limits_{r\to\infty}\mathsf P_{\alpha,\mathsf Q}\big[Q_{re_1}>s\delta_r\,|\,re_1\text{ is visible from }0\big] = \exp\big(-\lambda_{\mathrm{BM}}s\big)
\]
with
\begin{equation}\label{eq:BM-lambda}
\lambda_{\mathrm{BM}} = \frac12\alpha(d-1)\kappa_{d-1}\mathsf E[\varrho^{d-2}],
\end{equation}
where $\varrho$ is a random variable with law $\mathsf Q$ (recall that $\mathsf E[\varrho^d]<\infty$).

\smallskip

By the definition \eqref{def:Qx} of $Q_{re_1}$, \eqref{def:BM} and \eqref{def:cones}, 
\begin{multline*}
\mathsf P_{\alpha,\mathsf Q}\big[Q_{re_1}>s\delta_r\,|\,re_1\text{ is visible from }0\big]
= 
\frac{\mathsf P_{\alpha,\mathsf Q}[\text{every $x\in B(re_1,s\delta_r)$ is visible from $0$}]}{\mathsf P_{\alpha,\mathsf Q}[\text{$re_1$ is visible from $0$}]}\\
=
\frac{\mathsf P[\mathcal B^\alpha_\mathsf Q\cap \ell_{re_1}^{s\delta_r}=\emptyset]}{\mathsf P[\mathcal B^\alpha_\mathsf Q\cap \ell_{re_1}=\emptyset]}
=
\exp\big(-\alpha\big[\mathsf E\big[\lambda_d\big(\ell_{re_1}^{s\delta_r}(\varrho)\big)\big] - \mathsf E\big[\lambda_d\big(\ell_{re_1}(\varrho)\big)\big]\big]\big).
\end{multline*}
Thus, it suffices to prove that 
\begin{equation}\label{eq:BM-proof-1}
\lim\limits_{r\to\infty}\big[\mathsf E\big[\lambda_d\big(\ell_{re_1}^{s\delta_r}(\varrho)\big)\big] - \mathsf E\big[\lambda_d\big(\ell_{re_1}(\varrho)\big)\big]\big] = \frac{\lambda_{\mathrm{BM}}}{\alpha}s.
\end{equation}
Note that for every $t>0$, 
\begin{equation}\label{eq:BM-volume-1}
\lambda_d\big(\ell_{re_1}(t)\big) = \kappa_dt^d + \kappa_{d-1}t^{d-1}r.
\end{equation}
Furthermore, if $\beta_r$ is the opening angle of  $\ell_{re_1}^{s\delta_r}(\rho)$, then $\sin\beta_r = \frac{s\delta_r}{r}=\frac{s}{r^2}$ and
\begin{eqnarray}\label{eq:BM-volume-2}
\lambda_d\big(\ell_{re_1}^{s\delta_r}(t)\big) 
&= 
&\kappa_dt^d + \kappa_{d-1}\int_0^r\Big(t+y\sin\beta_r\Big)^{d-1}dy + R_1\\ \nonumber
&= &\kappa_dt^d + \kappa_{d-1}t^{d-1}r + \kappa_{d-1}(d-1)t^{d-2}\frac{r^2}{2}\sin\beta_r + R_2\\ \nonumber
&= &\kappa_dt^d + \kappa_{d-1}t^{d-1}r + \frac12\kappa_{d-1}(d-1)t^{d-2}s + R_2,
\end{eqnarray}
where $|R_1|, |R_2|\leq C\max(t^{d-1},1)\delta_r$ for some $C=C(d,s)$, from which \eqref{eq:BM-proof-1} follows. \qed

\section{Proof of Theorem~\ref{thm:main-result} for the Poisson cylinders}\label{sec:PC}

Fix positive $\alpha$ and $\rho$ and denote by $\mathsf P_{\alpha,\rho}$ the law of the Poisson cylinders at level $\alpha$ with radius $\rho$. 
Recall that $\delta_r = 1$ for $d=2$ and $\delta_r = \frac{1}{r}$ for $d\geq 3$. 
We aim to prove that for every $s>0$, 
\[
\lim\limits_{r\to\infty}\mathsf P_{\alpha,\rho}\big[Q_{re_1}>s\delta_r\,|\,re_1\text{ is visible from }0\big] = \exp\big(-\lambda_{\mathrm{PC}}s\big)
\]
with
\begin{equation}\label{eq:PC-lambda}
\lambda_{\mathrm{PC}}= 
\left\{\begin{array}{ll}\alpha & d=2\\[4pt]
\frac12 \alpha(d-2)\kappa_{d-2}\rho^{d-3}\mathsf E[\|\xi\|]
& d\geq 3\end{array}\right.
\end{equation}
where $\xi$ is the orthogonal projection of a uniformly distributed point on the unit sphere in $\R^d$ onto the hyperplane $\{x=(x_1,\ldots, x_d)\in\R^d\,:\,x_1=0\}$ and 
$\mathsf E[\|\xi\|]
= \frac{\mathrm{Beta}(\frac12,\frac d2)}{\mathrm{Beta}(\frac12,\frac{d-1}2)}$.

\medskip

By the definition \eqref{def:Qx} of $Q_r$, \eqref{def:PC-PL}, \eqref{def:PC-alpha-rho} and \eqref{def:cones}, 
\begin{multline*}
\mathsf P_{\alpha,\rho}\big[Q_{re_1}>s\delta_r\,|\,re_1\text{ is visible from }0\big]
= 
\frac{\mathsf P_{\alpha,\rho}[\text{every $x\in B(re_1,s\delta_r)$ is visible from $0$}]}{\mathsf P_{\alpha,\rho}[\text{$re_1$ is visible from $0$}]}\\
=
\frac{\mathsf P[\mathcal L^\alpha_\rho\cap \ell_{re_1}^{s\delta_r}=\emptyset]}{\mathsf P[\mathcal L^\alpha_\rho\cap \ell_{re_1}=\emptyset]}
=
\frac{\mathsf P[\mathcal L^\alpha\cap \ell_{re_1}^{s\delta_r}(\rho)=\emptyset]}{\mathsf P[\mathcal L^\alpha\cap \ell_{re_1}(\rho)=\emptyset]}
=
\exp\big(-\alpha\big[\mu\big(\ell_{re_1}^{s\delta_r}(\rho)\big) - \mu\big(\ell_{re_1}(\rho)\big)\big]\big).
\end{multline*}
Thus, it suffices to prove that 
\begin{equation}\label{eq:PC-proof-1}
\lim\limits_{r\to\infty}\big[\mu\big(\ell_{re_1}^{s\delta_r}(\rho)\big) - \mu\big(\ell_{re_1}(\rho)\big)\big] = \frac{\lambda_{\mathrm{PC}}}{\alpha}s.
\end{equation}
By the definition \eqref{def:PC-mu} of $\mu$, 
\[
\mu\big(\ell_{re_1}^{s\delta_r}(\rho)\big) - \mu\big(\ell_{re_1}(\rho)\big)
= 
\int_{SO_d}\big[\lambda_{d-1}\big(\pi(\ell_{r\phi(e_1)}^{s\delta_r}(\rho))\big) - \lambda_{d-1}\big(\pi(\ell_{r\phi(e_1)}(\rho))\big)\big]\,\nu(d\phi).
\]
If $d=2$, then $\lambda_{d-1}\big(\pi(\ell_{r\phi(e_1)}^{s\delta_r}(\rho))\big) - \lambda_{d-1}\big(\pi(\ell_{r\phi(e_1)}(\rho))\big) = s\delta_r = s$, so \eqref{eq:PC-proof-1} holds. 

\smallskip

Let $d\geq 3$. Note that $\pi(\ell_{r\phi(e_1)}^{s\delta_r}(\rho)) = \ell^{s\delta_r}_{r\pi(\phi(e_1))}(\rho)\cap H$ and $\pi(\ell_{r\phi(e_1)}(\rho)) = \ell_{r\pi(\phi(e_1))}(\rho)\cap H$, where $H$ be the hyperplane $\{x=(x_1,\ldots, x_d)\in\R^d\,:\,x_1=0\}$. 
Furthermore, if $\phi$ has law $\nu$, then $\phi(e_1)$ is uniformly distributed on the unit sphere $S$ in $\R^d$ centered at $0$. 
Thus, if $\xi$ is the orthogonal projection on $H$ of a uniformly distributed point on $S$, then 
\[
\mu\big(\ell_{re_1}^{s\delta_r}(\rho)\big) - \mu\big(\ell_{re_1}(\rho)\big)
= \mathsf E\big[
\lambda_{d-1}\big(\ell_{r\xi}^{s\delta_r}(\rho)\cap H\big) - \lambda_{d-1}\big(\ell_{r\xi}(\rho)\cap H\big)
\big].
\]
Similarly to the volume formulas \eqref{eq:BM-volume-1} and \eqref{eq:BM-volume-2}, we obtain
\[
\lambda_{d-1}\big(\ell_{r\xi}(\rho)\cap H\big) =  \kappa_{d-1} \rho^{d-1} + \kappa_{d-2}\rho^{d-2}r\|\xi\|
\]
and 
\[
\lambda_{d-1}\big(\ell_{r\xi}^{s\delta_r}(\rho)\cap H\big) = 
\kappa_{d-1}\rho^{d-1} + \kappa_{d-2}\rho^{d-2}r\|\xi\| + \frac12\kappa_{d-2}(d-2)\rho^{d-3}s\|\xi\| + R,
\]
where $|R|\leq C\delta_r$ for some $C=C(d,s,\rho)$. Thus, 
\[
\lim\limits_{r\to\infty}\big[\mu\big(\ell_{re_1}^{s\delta_r}(\rho)\big) - \mu\big(\ell_{re_1}(\rho)\big)\big]
=
\frac12\kappa_{d-2}(d-2)\rho^{d-3}s\,\mathsf E\big[\|\xi\|\big],
\]
which proves \eqref{eq:PC-proof-1}. 
Finally, we compute $\mathsf E[\|\xi\|]$. 
Using spherical coordinates $\varphi_1,\ldots, \varphi_{d-1}$, where $\varphi_1$ is the angle with $e_1$, $\varphi_1,\ldots,\varphi_{d-2}\in[0,\pi]$ and $\varphi_{d-1}\in[0,2\pi]$, we obtain that 
\begin{eqnarray*}
\mathsf E\big[\|\xi\|\big] &= 
&\frac{\int \big(\sin\varphi_1\big)\,\big(\sin^{d-2}\varphi_1\,\ldots\,\sin\varphi_{d-2}\big)d\varphi_1\ldots d\varphi_{d-1}}{\int \big(\sin^{d-2}\varphi_1\,\ldots\,\sin\varphi_{d-2}\big)d\varphi_1\ldots d\varphi_{d-1}}\\
&= &\frac{\int_0^\pi \sin^{d-1}\varphi_1\,d\varphi_1}{\int_0^\pi \sin^{d-2}\varphi_1\,d\varphi_1} 
= \frac{\mathrm{Beta}(\frac12,\frac d2)}{\mathrm{Beta}(\frac12,\frac{d-1}2)},
\end{eqnarray*}
see e.g.\ \cite[6.2.1]{AS:math-functions}. 
The proof is completed.
\qed

\section*{Acknowledgements}
We thank Jean-Baptiste Gou\'er\'e for suggesting us to investigate the conditional distribution of $Q_x/\delta_{\|x\|}$. 
The research of both authors has been supported by the DFG Priority Program 2265 ``Random Geometric Systems'' (Project number 443849139).

\end{document}